\documentclass[12pt]{article}
\usepackage[english]{babel}
\usepackage{amssymb, amsfonts, amsthm, amsmath}

\usepackage{mathrsfs}
\usepackage{graphicx}
\usepackage{wrapfig}
\usepackage{caption2}

\newtheorem{theorem}{\bf Theorem}
\newtheorem{Remark}{\bf Remark}

\title{\textbf{\large Solvability Issues of Two Classes of Quasilinear and Nonlinear Integral Equations with a Sum-Difference Kernel on the Half-Line \footnote{The research of Khachatur A. Khachatryan was supported by the Higher Education and Science Committee of the Republic of Armenia (Grant No. 23RL-1A027).}}}
\author{A.Kh. Khachatryan, Kh.A. Khachatryan, H.S. Petrosyan}

\begin{document}
\date{}
\maketitle

\begin{abstract}
This paper investigates two classes of quasilinear and essentially nonlinear integral equations with a sum-difference kernel on the half-line. Such equations arise in various areas of physics, including the theory of radiative transfer in spectral lines, the dynamic theory of p-adic strings, and the kinetic theory of gases within the framework of the modified Bhatnagar–Gross–Krook (BGK) model.
Under specific conditions on the kernel and the nonlinear terms, we establish constructive existence theorems for non-negative, nontrivial, and continuous solutions. For the quasilinear case, we construct a one-parameter family of non-negative, nontrivial, linearly growing, and continuous solutions. For the class of essentially nonlinear equations, we prove the uniform convergence—at a geometric rate—of a specially constructed sequence of successive approximations to a non-negative, nontrivial, continuous, and bounded solution.
We also study the asymptotic behavior of the constructed solutions at infinity. Additionally, for the second class of equations, a uniqueness theorem is established within a certain subclass of non-negative, bounded, and nontrivial functions. The paper concludes with concrete examples of kernels and nonlinearities that satisfy all conditions of the proven theorems.
\end{abstract}
\textbf{References:} 22.

\section{Introduction}

\subsection{Setting of problem. Basic conditions.}
Consider the following classes of quasilinear and essentially nonlinear integral equations on the half-axis:
\begin{equation}\label{Khachatryan1.1}
    f(x) = \int\limits_0^{\infty} \big( K(x - t) - K(x + t) \big) \big( f(t) + \omega_1(t, f(t)) \big) dt, \quad x \in \mathbb{R}^+: = [0, +\infty),
\end{equation}
\begin{equation}\label{Khachatryan1.2}
    Q(B(x)) = \int\limits_0^{\infty} \big( K(x - t) - K(x + t) \big) \big( B(t) + \omega_2(t, B(t)) \big) dt, \quad x \in \mathbb{R}^+,
\end{equation}
with respect to the sought nonnegative and continuous on  $\mathbb{R}^+$ functions $f(x)$ and $B(x)$ respectively. In equations  \eqref{Khachatryan1.1} and \eqref{Khachatryan1.2} the kernel  $K$ satisfies the following conditions:
\begin{enumerate}
    \item[1)] $K(x) > 0, \, x \in \mathbb{R} := (-\infty, +\infty), \, K(-t) = K(t), \, t \in \mathbb{R}^+; \, K \in C^1(\mathbb{R}) \cap M^1(\mathbb{R}),$ \\
    where $C^1(\mathbb{R})$ --- is the space of continuously differentiable functions on the set  $\mathbb{R}$, and $M^1(\mathbb{R})$ --- is the space of bounded functions with their first derivative on the set $\mathbb{R},$
    \item[2)] $K(x)$ monotonically decreases on the set  $\mathbb{R}^+,$ $\int\limits_0^\infty K(t)dt=\frac{1}{2},$
    \item[3)] $\int\limits_{0}^{\infty}x K(x) dx < +\infty.$
\end{enumerate}

The nonlinearities $Q$, $\omega_1$ and $\omega_2$ satisfy the criticality conditions:
\begin{equation}\label{Khachatryan1.3}
    Q(0) = 0, \quad \omega_1(t, 0) = \omega_2(t, 0) \equiv 0 \quad \text{on } \mathbb{R}^+
\end{equation}
and have the following properties:
\begin{itemize}
    \item[$a_1)$] $\omega_1 \in C(\mathbb{R}^+ \times \mathbb{R}^+)$ and for each fixed  $t \in \mathbb{R}^+$ the function  $\omega_1(t, u)$ increases monotonically in  $u$ on $\mathbb{R}^+,$
    \item[$a_2)$] there exists  $\sup\limits_{u\in\mathbb{R}^+} \omega_1(t, u) =: \mu_1(t)$, where $\mu_1 \in C_0(\mathbb{R}^+)$ and $\int\limits_0^\infty t \mu_1(t) dt < +\infty,$ \\
     where $C_0(\mathbb{R}^+)$ --- is the space of continuous functions on  $\mathbb{R}^+$, that have a zero
limit at infinity,
    \item[$b_1)$] $\omega_2 \in C(\mathbb{R}^+ \times \mathbb{R}^+),$  for each fixed  $t \in \mathbb{R}^+$ the function  $\omega_2(t, u)$ increases monotonically in  $u$ on $\mathbb{R}^+$ and is concave on  $\mathbb{R}^+,$
    \item[$b_2)$] there exists a positive and summable on  $\mathbb{R}^+$ function  $\mu_2(t)$ with zero limit at infinity such that
 \[\omega_2(t,u)\leq \mu_2(t)u,\  t\in \mathbb{R}^+, u\in \mathbb{R}^+,\]
    \item[I)] $Q \in C(\mathbb{R}^+)$ and $Q(u)$ increases monotonically on the set  $\mathbb{R}^+,$
    \item[II)] $Q(u)$ is convex on  $\mathbb{R}^+$ and there exists a number $\xi > 0$ such that
    $
    Q(\xi) = (1 + M)\xi ,$ where $M:= \int\limits_0^\infty  \mu_2(t)dt \sup\limits_{x \in \mathbb{R}} K(x),
    $
        \item[III)] there exists a continuous monotonically increasing and concave mapping $\varphi \colon [0,1] \to [0,1]$ with the properties  $\varphi(0) = 0$, $\varphi(1) = 1$ such that
    \[
    G(\sigma u) \geq \varphi(\sigma)G(u), \quad \sigma \in (0,1), \quad u \in (0, (1+M)\xi),
    \]
    where $G(u)$ --- is the inverse function of the function  $Q$ on $\mathbb{R}^+.$
\end{itemize}

Note that the criticality condition \eqref{Khachatryan1.3} ensures the existence of trivial (zero) solutions for equations \eqref{Khachatryan1.1} and \eqref{Khachatryan1.2}. The main goal of this work is to construct nontrivial non-negative continuous solutions in certain functional spaces and to study their asymptotic behavior at infinity.

\subsection{Possible Applications and Background.}

Equations \eqref{Khachatryan1.1} and \eqref{Khachatryan1.2}, in addition to their purely mathematical interest, are of particular applied significance. Such equations arise in the nonlinear theory of radiation transfer in spectral lines, in the dynamic theory of $p$-adic open and open-closed strings for scalar tachyon fields, and in the kinetic theory of gases and plasma within the framework of the modified Bhatnagar–Gross–Krook model (see \cite{eng1}-\cite{cerc4}). Moreover, equation \eqref{Khachatryan1.2} appears in the mathematical theory of the geographical spread of epidemics within the framework of the modified Diekmann-Kaper model (see \cite{diek5}, \cite{khach6}). In the mentioned theories, the question of constructive solvability of equations \eqref{Khachatryan1.1} and \eqref{Khachatryan1.2} is of particular importance for building approximate solutions and studying certain qualitative properties of the constructed solutions.

In the case where the kernel \( K \) in \eqref{Khachatryan1.1} depends only on the difference of its arguments, and under various constraints on the function \(\omega_1\) , equation \eqref{Khachatryan1.1} has been studied in works \cite{arb7}-\cite{khach9}. When \(\omega_2 \equiv 0\) , issues of existence, uniqueness, and some qualitative properties of the constructed solution to equation \eqref{Khachatryan1.2} have been studied in detail in \cite{vla3}, \cite{juko10}-\cite{petr13}. It should also be noted that equation \eqref{Khachatryan1.2} was examined in \cite{khach14} in the case where \(\omega_2(t, u) = \mu_2(t)u\), due to its significant applied value in the theory of \(p\)-adic strings.
It is also noteworthy that equations related to \eqref{Khachatryan1.2} on the entire real line with difference kernels and \(\omega_2 \equiv 0\)  have been studied in works \cite{vla3} and \cite{khach9}-\cite{khach17}.

\subsection{Summary of Main Results.}

The first part of this work is devoted to the study and solution of the quasilinear equation \eqref{Khachatryan1.1}. A theorem is proved on the existence of a one-parameter family of non-negative, nontrivial, continuous on $\mathbb{R}^+,$ and linearly growing solutions of equation \eqref{Khachatryan1.1}. The set of corresponding parameters is described, and an explicit asymptotic formula for the constructed solutions at infinity is established.

The second part focuses on the nonlinear integral equation \eqref{Khachatryan1.2}. First, a constructive theorem is proven regarding the existence of a non-negative, nontrivial, bounded, and continuous solution on $\mathbb{R}^+.$ Then, the uniform convergence of the introduced sequence of successive approximations to the solution is shown, with a rate following an infinitely decreasing geometric progression.

The third part addresses the uniqueness of the solution to equation \eqref{Khachatryan1.2}. In a certain subclass of non-negative and bounded functions on $\mathbb{R}^+,$  the uniqueness of the constructed solution is proved. At the end of the work, specific examples of the kernel $K$  and the nonlinearities $Q, \omega_j, j=1,2$ , are presented to illustrate the importance of the obtained results.

\section{Solvability of Equations  \eqref{Khachatryan1.1} and \eqref{Khachatryan1.2}}

\subsection{One-Parameter Family of Solutions to Equation   \eqref{Khachatryan1.1}.}
\setcounter{equation}{0}
We begin with the following result:

\begin{theorem} Under conditions  $1)-3)$, $a_1)$, and $a_2)$, the quasilinear integral equation \eqref{Khachatryan1.1} has a one-parameter family of non-negative, continuous, and linearly increasing solutions on $\mathbb{R}^+.$ Moreover, for each such solution $\{f_\gamma(x)\}_{\gamma \in (0, +\infty)},$ the following asymptotic formula holds:
\begin{equation}\label{Khachatryan2.1}
\lim_{x \to +\infty} \frac{f_\gamma(x)}{x} = \gamma.
\end{equation}
\end{theorem}
\begin{proof} First, note that from conditions  $1)$, $2)$  it immediately follows that
\begin{equation}\label{Khachatryan2.2}
K(x - t) > K(x + t), \quad \text{at } (x, t) \in (0, +\infty) \times (0, +\infty).
\end{equation}

Now, along with equation  \eqref{Khachatryan1.1} we consider the following auxiliary linear integral equation with a sum-difference kernel:
\begin{equation}\label{Khachatryan2.3}
\psi(x) = g(x) + \int\limits_0^\infty (K(x - t) - K(x + t)) \psi(t) \, dt, \quad x \in \mathbb{R}^+,
\end{equation}
with respect to the desired non-negative and continuous on  $\mathbb{R}^+$ function $\psi(x)$. In equation  \eqref{Khachatryan2.3} the free term  $g(x)$ admits the following representation:
\[
g(x) = \int\limits_0^\infty \left(K(x - t) - K(x + t)\right) \mu_1(t) \, dt, \quad x \in \mathbb{R}^+.
\]

Taking into account inequality  \eqref{Khachatryan2.2}, conditions $1)-3)$, $a_2)$ and using Fubini's theorem (see  \cite{kol18}), it is easy to verify the following facts for the free term  $g:$
\begin{equation}\label{Khachatryan2.4}
g(x) > 0, \quad x \in (0, +\infty), \quad g(0) = 0, \quad g \in C_M(\mathbb{R}^+) \cap L_1(\mathbb{R}^+), \quad x g(x) \in L_1(\mathbb{R}^+),
\end{equation}
where $C_M(\mathbb{R}^+)$ is the space of continuous and bounded functions on  $\mathbb{R}^+.$

Let us now consider the following iterations for equation \eqref{Khachatryan2.3}:
\begin{equation}\label{Khachatryan2.5}
\begin{array}{c}
\displaystyle\psi_{n+1}(x) = g(x) + \int\limits_0^\infty \left(K(x - t) - K(x + t)\right) \psi_n(t) \, dt, \\
\displaystyle \psi_0(x) = g(x), \, n = 0, 1, 2, \ldots, \quad x \in \mathbb{R}^+.
 \end{array}
\end{equation}

Using conditions  $1)$–$3)$ and properties  \eqref{Khachatryan2.4} of the function  $g$, it is easy to check that the following properties of the sequence $\{\psi_n(x)\}_{n=0}^\infty$ are satisfied:
\begin{equation}\label{Khachatryan2.6}
\psi_n\in C(\mathbb{R}^+), \quad \psi_{n+1}(x)\geq \psi_n(x),\quad n=0,1,2,\ldots, \quad x\in \mathbb{R}^+,
\end{equation}
\begin{equation}\label{Khachatryan2.7}
\psi_n(x) \leq H(x), \quad n = 0, 1, 2, \ldots, \; x \in \mathbb{R}^+,
\end{equation}
where $ H(x) $ — is a continuous positive and bounded on  $ \mathbb{R}^+ $ solution of the following
Wiener-Hopf integral equation (see \cite{arb19}):
\begin{equation}\label{Khachatryan2.8}
H(x) = g(x) + \int\limits_0^\infty K(x-t) H(t) dt, \quad x \in \mathbb{R}^+.
\end{equation}

From \eqref{Khachatryan2.6} and \eqref{Khachatryan2.7} it immediately follows that the sequence of continuous functions  $ \{\psi_n(x)\}_{n=0}^\infty : \lim\limits_{n \to \infty} \psi_n(x) = \psi(x),$
converges pointwise, and the limit function $ \psi(x) $, according to B. Levy's theorem (see  \cite{kol18}), satisfies equation  \eqref{Khachatryan2.3}.
We also note that from  \eqref{Khachatryan2.6}, \eqref{Khachatryan2.7} and the above-listed properties of the solution   $H(x)$ of equation  \eqref{Khachatryan2.8} it follows that the limit function  $ \psi(x) $  is bounded, and the following two-sided inequality holds:
\begin{equation}\label{Khachatryan2.9}
g(x) \leq \psi(x) \leq H(x), \quad x \in \mathbb{R}^+.
\end{equation}

We now proceed to the study of the quasilinear equation  \eqref{Khachatryan1.1} and introduce the following special successive approximations for any $\gamma \in (0, +\infty)$:
\begin{equation}\label{Khachatryan2.10}
\begin{array}{c}
\displaystyle f_{n+1}^\gamma(x) = \int\limits_0^\infty \big(K(x-t) - K(x+t)\big) \big(f_n^\gamma(t) + \omega_1(t, f_n^\gamma(t))\big) dt,\\
\displaystyle f_0^\gamma(x) = \gamma x, \quad n = 0, 1, 2, \ldots, \; x \in \mathbb{R}^+.
\end{array}
\end{equation}

Using the continuity of the functions  $ K $ and $\omega_1$ we can easily verify by induction on $n$ that
\begin{equation}\label{Khachatryan2.11}
f_n^\gamma \in C(\mathbb{R}^+), \quad n = 0, 1, 2, \ldots, \; \gamma \in (0, +\infty).
\end{equation}

We now prove that
\begin{equation}\label{Khachatryan2.12}
f_{n+1}^\gamma(x) \geq f_n^\gamma(x), \quad n = 0, 1, 2, \ldots, \; x \in \mathbb{R}^+, \, \gamma \in (0, +\infty).
\end{equation}

First, we show the validity of the inequality  $ f_1^\gamma(x) \geq f_0^\gamma(x) $, $ x \in \mathbb{R}^+ $, $\gamma\in(0,+\infty).$ Indeed, using conditions  $1), 2), 3)$ and $a_1),$ as well as inequality  \eqref{Khachatryan2.2}, from \eqref{Khachatryan2.10} we will have
\[f_1^\gamma (x)\geq \gamma  \int\limits_0^\infty \big(K(x-t) - K(x+t)\big)tdt=\gamma \int\limits_{-\infty}^x K(y)(y-x)dy-\gamma \int\limits_x^\infty K(y)(y-x)dy=\]
\[=\gamma x \int\limits_{-\infty}^x K(y)dy- \gamma  \int\limits_{-\infty}^x K(y)ydy- \gamma  \int\limits_{x}^\infty K(y)ydy+\gamma x \int\limits_{x}^\infty K(y)dy=\]
\[= \gamma x - \int\limits_{-\infty}^\infty y K(y)dy=\gamma x=f_0^\gamma (x).\]
Assuming that  $ f_n^\gamma(x) \geq f_{n-1}^\gamma(x), \; x \in \mathbb{R}^+, \; \gamma \in (0,+ \infty) $ for some natural  $ n $ and again using condition  $a_1)$ from \eqref{Khachatryan2.10} by virtue of inequality \eqref{Khachatryan2.2}, we obtain that $ f_{n+1}^\gamma(x)\geq f_n^\gamma(x), \; x \in \mathbb{R}^+, \; \gamma \in (0, +\infty) $.

Let us prove the validity of the following inequality from above:
\begin{equation}\label{Khachatryan2.13}
f_{n}^\gamma(x) \leq \gamma x + \psi(x), \; x \in \mathbb{R}^+, \; \gamma \in (0, +\infty).
\end{equation}
In the case $ n = 0 $, inequality \eqref{Khachatryan2.13} follows directly from the definition of the zeroth approximation in the iterative scheme \eqref{Khachatryan2.10}, due to relations \eqref{Khachatryan2.4} and \eqref{Khachatryan2.9}.
Assume now that inequality \eqref{Khachatryan2.13} holds for some natural number $ n $  and for all $ x \in \mathbb{R}^+ $  and $ \gamma \in (0, +\infty) $. Then, using conditions 1)–3), $a_2)$, and inequality \eqref{Khachatryan2.2}, we obtain:
\[
f_{n+1}^\gamma(x) \leq \int\limits_0^\infty \big(K(x-t) - K(x+t)\big) \big(\gamma t + \psi(t) + \omega_1(t, \gamma t + \psi(t))\big) dt \leq
\]
\[
\leq \gamma x + \int\limits_0^\infty \big(K(x-t) - K(x+t)\big) \psi(t) dt + g(x) = \gamma x + \psi(x), \; x \in \mathbb{R}^+, \; \gamma \in (0,+\infty).
\]

Using similar reasoning and taking into account the monotonicity of the function  $ \omega_1(t, u) $ with respect to $ u $, it is also easy to verify by induction on  $n$ that the following inequality from below holds for all  $ \gamma_1, \gamma_2 \in (0, +\infty), \; \gamma_1 > \gamma_2 $:
\begin{equation}\label{Khachatryan2.14}
f_n^{\gamma_1}(x) - f_n^{\gamma_2}(x)\geq (\gamma_1-\gamma_2)x, \; n = 0, 1, \ldots, \; x \in \mathbb{R}^+.
\end{equation}

Thus, based on  \eqref{Khachatryan2.11}- \eqref{Khachatryan2.14} we conclude that for any  $ \gamma \in (0,+\infty) $ the sequence of continuous  on  $ \mathbb{R}^+ $  functions $\{f_n^\gamma(x)\}_{n=0}^\infty $  has a pointwise limit when $ n \to \infty :
\lim\limits_{n \to \infty} f_n^\gamma(x) = f_\gamma(x), \; x \in \mathbb{R}^+.
$ Taking into consideration the continuity of the functions  $ K $ and $ \omega_1 $, inequality \eqref{Khachatryan2.2}, and also \eqref{Khachatryan2.11}–\eqref{Khachatryan2.13}, according to B. Levy's theorem, we can assert that the limit function satisfies equation \eqref{Khachatryan1.1}. From \eqref{Khachatryan2.11}–\eqref{Khachatryan2.14} it follows that
\begin{equation}\label{Khachatryan2.15}
\gamma x \leq f_\gamma(x) \leq \gamma x + \psi(x), \; x \in \mathbb{R}^+, \; \gamma \in (0, +\infty),
\end{equation}
\begin{equation}\label{Khachatryan2.16}
f_{\gamma_1}(x)-f_{\gamma_2}(x)\geq (\gamma_1-\gamma_2)x,\quad \text{at}\quad \gamma_1>\gamma_2>0,\quad x\in\mathbb{R}^+.
\end{equation}
From the continuity of the function  $ K $ taking into account conditions  3) it follows that
\[
f_\gamma \in C(\mathbb{R}^+), \; \gamma \in (0, +\infty).
\]

Since $ \psi \in M(\mathbb{R}^+),$   the limit relation  \eqref{Khachatryan2.1} immediately follows from  \eqref{Khachatryan2.15}. Thus, the theorem is proved.
\end{proof}
\begin{Remark}
In the process of proving Theorem 1, additional properties \eqref{Khachatryan2.15} and \eqref{Khachatryan2.16} of the constructed solutions were also established.
\end{Remark}

\subsection{Constructive Solvability of Equation   \eqref{Khachatryan1.2}.}

Now we turn to the analysis of equation \eqref{Khachatryan1.2}. Alongside this equation, we also consider the following auxiliary nonlinear integral equations:
\begin{equation}\label{Khachatryan2.17}
Q(F(x)) = \int\limits_0^\infty \big(K(x-t) - K(x+t)\big) F(t) dt, \; x \in \mathbb{R}^+,
\end{equation}
\begin{equation}\label{Khachatryan2.18}
Q(\Phi(x)) = \int\limits_0^\infty \big(K(x-t) - K(x+t)\big) \big(1 + \mu_2(t)\big) \Phi(t) dt, \; x \in \mathbb{R}^+,
\end{equation}
with respect to the sought nonnegative bounded and continuous on  $ \mathbb{R}^+ $ functions $ F $ and $ \Phi $ respectively. It follows from the results of  \cite{khach12} that equation  \eqref{Khachatryan2.17} has a unique non-negative non-trivial monotonically non-decreasing continuous and bounded on  $ \mathbb{R}^+ $ solution  $ F(x) $, and
\begin{equation}\label{Khachatryan2.19}
F(0) = 0, \; F(x) > 0 \; \text{at} \; x > 0, \; \lim_{x \to +\infty} F(x) = \eta, \; \eta - F \in L_1(\mathbb{R}^+),
\end{equation}
where the number  $ \eta > 0 $ is a positive fixed point of the mapping  $ y = Q(u) $. The existence of such a point follows immediately from properties  I), II) of the function $ Q(u) $. According to the results of \cite{khach14}, equation  \eqref{Khachatryan2.18} has a unique non-negative non-trivial continuous and bounded on $ \mathbb{R}^+ $ solution $ \Phi(x) $, where
\begin{equation}\label{Khachatryan2.20}
\Phi(0) = 0, \; \xi>\Phi(x) > F(x) \; \text{at} \; x > 0, \; \lim_{x \to +\infty} \Phi(x) = \eta, \; \eta- \Phi \in L_1(\mathbb{R}^+).
\end{equation}

Let us now consider the following iterations for equation \eqref{Khachatryan2.1}:
\begin{equation}\label{Khachatryan2.21}
\begin{array}{c}
\displaystyle Q(B_{n+1}(x)) = \int\limits_0^\infty \big(K(x-t) - K(x+t)\big) \big(B_n(t) + \omega_2(t, B_n(t))\big) dt, \\
\displaystyle B_0(x) = \Phi(x), \; n = 0, 1, 2, \ldots, \; x \in \mathbb{R}^+.
\end{array}
\end{equation}

Using conditions 1)–3), I), II), $b_1)$ and $b_2)$, by induction on  $ n $ it is easy to verify the following facts:
\begin{equation}\label{Khachatryan2.22}
B_n \in C(\mathbb{R}^+), \; n = 0, 1, 2, \ldots,
\end{equation}
\begin{equation}\label{Khachatryan2.23}
B_{n+1}(x) \leq B_n(x), \; n = 0, 1, 2, \ldots, \; x \in \mathbb{R}^+.
\end{equation}
\begin{equation}\label{Khachatryan2.24}
B_n(x)\geq F(x), \; n = 0, 1, 2, \ldots, \; x \in \mathbb{R}^+.
\end{equation}
From \eqref{Khachatryan2.22}, \eqref{Khachatryan2.23} and \eqref{Khachatryan2.24}, the pointwise convergence of the successive approximations in \eqref{Khachatryan2.21} follows. In fact, by additionally using condition III), we will establish the uniform convergence of the sequence of continuous functions $\{B_n(x)\}_{n=0}^\infty$ to a continuous solution $B(x)$ of equation \eqref{Khachatryan1.2} on $\mathbb{R}^+$  with a convergence rate governed by a decreasing geometric progression.

To this end, we introduce the following function on the set $(0, +\infty)$:
\begin{equation}\label{Khachatryan2.25}
\chi(x) := \frac{\int\limits_{0}^\infty \big(K(x-t) - K(x+t)\big) \big(\Phi(t) + \omega_2(t, \Phi(t))\big) dt}
{\int\limits_{0}^\infty K(x-t) - K(x+t) \big(1 + \mu_2(t) \Phi(t)\big) dt}.
\end{equation}
Since the convolution of a bounded and integrable function is continuous (see \cite{rud20}), it follows from properties $1), 2), b_1), $ and $b_2)$ that:
\[
\chi \in C(0, +\infty).
\]
Below we use the following well-known limit relation in the convolution operation
(see \cite{gev21}):
\begin{enumerate}
  \item[i)]  if $W_1(x) \geq 0$, $x\in\mathbb{R},$ $W_2(t)\geq0,$ $t\in\mathbb{R}^+,$ $W_1 \in L_1(\mathbb{R})$, $W_2 \in L_1(\mathbb{R}^+)$, $\lim\limits_{x\rightarrow+\infty}W_2(x)=0$, then
\begin{equation*}\label{Khachatryan2.26}
\lim_{x \to +\infty} \int\limits_{0}^\infty W_1(x-t) W_2(t) dt =0,
\end{equation*}
  \item[j)] if $W_1(x) \geq 0$, $x\in\mathbb{R},$ $W_2(t)\geq0,$ $t\in\mathbb{R}^+,$  $W_1 \in L_1(\mathbb{R})$, $W_2 \in L_\infty(\mathbb{R}^+)$, $\lim\limits_{x\rightarrow+\infty}W_2(x) < +\infty$,  then
$$
\lim_{x \to +\infty} \int\limits_{0}^\infty W_1(x-t) W_2(t) dt = \int\limits_{-\infty}^\infty W_1(\tau) d\tau \cdot \lim_{x \to +\infty} W_2(x).
$$
\end{enumerate}

Taking into account  \eqref{Khachatryan2.20}, $i), j)$, and conditions  $1), 2), b_1)$ and  $b_2)$ from \eqref{Khachatryan2.25} we have:

\begin{equation}\label{Khachatryan2.27}
\begin{array}{c}
\displaystyle\chi(+\infty) =
\frac{\lim\limits_{x \to +\infty} \int\limits_{0}^\infty \big(K(x-t) - K(x+t)\big)\big(\Phi(t) + \omega_2(t, \Phi(t))\big) dt}
{\lim\limits_{x \to +\infty} \int\limits_{0}^\infty \big(K(x-t) - K(x+t)\big)\big(1 + \mu_2(t) \Phi(t)\big) dt}=\\
\displaystyle =
\frac{\int\limits_{-\infty}^\infty K(y) dy \cdot \lim\limits_{t \to +\infty} \big(\Phi(t) + \omega_2(t, \Phi(t))\big)}
{\int\limits_{-\infty}^\infty K(y) dy \cdot \lim\limits_{t \to +\infty} \big(\Phi(t) + \mu_2(t) \Phi(t)\big)} =
\frac{\eta + \lim\limits_{t \to +\infty} \omega_2(t, \Phi(t))}
{\eta + \lim\limits_{t \to +\infty} \big(\mu_2(t) \Phi(t)\big)} = 1.
\end{array}
\end{equation}

From $\eqref{Khachatryan2.2}, \eqref{Khachatryan2.20}$ and $b_1)$ it immediately follows that

\begin{equation}\label{Khachatryan2.28}
\chi(x) > 0 \quad \text{ at }\quad x > 0.
\end{equation}
Now we investigate the limit value of the function  $\chi(x)$, when  $x \to 0^+$. Using conditions  $1), 2), b_1)$ and $b_2)$, according to L'Hôpital's rule from  \eqref{Khachatryan2.25}, we obtain the following expressions for  $\chi(+0):=\lim\limits_{x\rightarrow0^+}\chi(x):$

\[
\chi(+0) =  \frac{\lim\limits_{x \to 0^+}\int\limits_{0}^\infty \big(K'_x(x-t) - K'_x(x+t)\big) (\Phi(t) + \omega_2(t, \Phi(t))) dt}
{\lim\limits_{x \to 0^+}\int\limits_{0}^\infty \big(K'_x(x-t) - K'_x(x+t)\big) (1 + \mu_2(t)\Phi(t)) dt} =
\]

\[
= \frac{\int\limits_{0}^\infty (K'(-t)-K'(t)) (\Phi(t) + \omega_2(t, \Phi(t))) dt}
{\int\limits_{0}^\infty (K'(-t)-K'(t)) (1 + \mu_2(t)\Phi(t)) dt} =
\]

\[
= \frac{-2 \int\limits_{0}^\infty K'(t) (\Phi(t) + \omega_2(t, \Phi(t))) dt}
{-2 \int\limits_{0}^\infty K'(t) (1 + \mu_2(t)\Phi(t)) dt}
\geq
\frac{-\int\limits_{0}^\infty K'(t) (F(t) + \omega_2(t, F(t))) dt}
{-\xi \int\limits_{0}^\infty K'(t) (1 + \mu_2(t)) dt} =: \varepsilon > 0,
\]

since $K'(t) < 0$, $t \in \mathbb{R}^+$, $F(0) = 0$, $F(t) > 0$, $t \in (0, +\infty)$, $\mu_2(t) \geq 0$, $t \in \mathbb{R}^+$.

On the other hand, since
\[
0 < F(t) + \omega_2(t, F(t)) \leq \eta + \mu_2(t)\eta< \xi(1 + \mu_2(t)), \quad t > 0,
\]

then $\varepsilon < 1$. It is also obvious that  $\chi(+0) \leq 1$, since  $\omega_2(t, \Phi(t)) \leq \mu_2(t)\Phi(t),\, t \in \mathbb{R}^+$.

From the above properties of the function  $\chi$ it follows that
$
\chi \in C(\mathbb{R}^+)$ and there exists a number   $ \sigma_0 \in (0, 1)$  such that:

\begin{equation}\label{Khachatryan2.29}
\sigma_0 \leq \chi(x) \leq 1, \quad x \in \mathbb{R}^+.
\end{equation}

Taking into consideration  \eqref{Khachatryan2.25}, \eqref{Khachatryan2.29} from the recurrence relations  \eqref{Khachatryan2.21} we obtain that
\begin{equation}\label{Khachatryan2.30}
\sigma_0 Q(B_0(x)) \leq Q(B_1(x)) \leq Q(B_0(x)), \quad x \in \mathbb{R}^+.
\end{equation}

Now we note that for each fixed $t \in \mathbb{R}^+$ the following estimation  holds:
\begin{equation}\label{Khachatryan2.31}
\omega_2(t, \delta u) \geq \delta \omega_2(t, u), \quad \delta \in (0, 1), \quad u \in \mathbb{R}^+.
\end{equation}

Indeed, for  $u = 0$ inequality  \eqref{Khachatryan2.31} turns into equality, because $\omega_2(t, 0) \equiv 0$ on $\mathbb{R}^+$.
Since  $\dfrac{\omega_2(t, u)}{u}$ monotonically decreases in  $u $ on $ (0, +\infty)$ (the latter follows from condition  $b_1$) and $0<\delta u<u,$ then
$\dfrac{\omega_2(t, \delta u)}{\delta u}> \dfrac{\omega_2(t, u)}{u}, \text{ whence we arrive at  \eqref{Khachatryan2.31}.}$

Now, using condition  $III)$ and the monotonicity of the function  $G$, it follows from  \eqref{Khachatryan2.30} that:
\begin{equation}\label{Khachatryan2.32}
\varphi(\sigma_0)B_0(t) \leq G(\sigma_0Q(B_0(t))) \leq B_1(t)\leq B_0(t), \quad t \in \mathbb{R}^+.
\end{equation}
Since $ \varphi(\sigma_0) \in (0, 1),$ then taking into account inequality \eqref{Khachatryan2.31} from estimate  \eqref{Khachatryan2.32} we arrive at a two-sided inequality:
\begin{equation}\label{Khachatryan2.33}
\begin{array}{c}
\displaystyle\varphi(\sigma_0)(B_0(t)+\omega_2(t, B_0(t)))\leq \varphi(\sigma_0) B_0(t) + \omega_2(t, \varphi(\sigma_0) B_0(t))\leq\\
\displaystyle\leq B_1(t) + \omega_2(t, B_1(t)) \leq B_0(t) + \omega_2(t, B_0(t)), \quad t \in \mathbb{R}^+.
\end{array}
\end{equation}
Taking into consideration inequality  \eqref{Khachatryan2.2} and recurrence relations \eqref{Khachatryan2.21} from \eqref{Khachatryan2.33} we obtain that
\[
\varphi(\sigma_0) Q(B_1(x)) \leq Q(B_2(x)) \leq Q(B_1(x)), \quad x \in \mathbb{R}^+,
\]
whence, by virtue of condition $III)$ and estimate  \eqref{Khachatryan2.31} it follows that
\begin{equation}\label{Khachatryan2.34}
\varphi(\varphi(\sigma_0)) (B_1(t) + \omega_2(t, B_1(t))) \leq B_2(t) + \omega_2(t, B_2(t)) \leq B_1(t) + \omega_2(t, B_1(t)), \quad t \in \mathbb{R}^+.
\end{equation}
Again, using  \eqref{Khachatryan2.2} from \eqref{Khachatryan2.34} and \eqref{Khachatryan2.21}, we arrive at the inequality:
\[
\varphi(\varphi(\sigma_0)) Q(B_2(x)) \leq Q(B_3(x)) \leq Q(B_2(x)), \quad x \in \mathbb{R}^+,
\]
whence, by virtue of $III)$ we obtain
\[
\varphi(\varphi(\varphi(\sigma_0))) B_2(x) \leq B_3(x) \leq B_2(x), \quad x \in \mathbb{R}^+.
\]
Continuing this procedure at the  $n$-th step, we arrive at the double inequality:
\begin{equation}\label{Khachatryan2.35}
\underbrace{\varphi(\varphi\ldots \varphi(\sigma_0))}_{n+1} B_n(x) \leq B_{n+1}(x) \leq B_n(x), \quad x \in \mathbb{R}^+.
\end{equation}

Let $\varepsilon_0 \in (0, 1)$ be an arbitrary number. Then, according to inequality \eqref{Khachatryan3.16} from \cite{khach22} the following estimate is valid:

\begin{equation}\label{Khachatryan2.36}
1 - \underbrace{\varphi(\varphi\ldots \varphi(\sigma_0))}_{n+1} \leq (1 - \sigma_0) k^{n+1}, \quad n = 0, 1, 2, \ldots,
\end{equation}
where
\begin{equation}\label{Khachatryan2.37}
k := k_{\varepsilon_0} = \frac{1 - \varphi(\varepsilon_0\sigma_0)}{1 - \varepsilon_0\sigma_0} \in (0, 1).
\end{equation}
Taking into account  \eqref{Khachatryan2.36}, \eqref{Khachatryan2.23}, \eqref{Khachatryan2.20}, it follows from  \eqref{Khachatryan2.35} that
\begin{equation}\label{Khachatryan2.38}
0 \leq B_n(x) - B_{n+1}(x) \leq \xi (1 - \sigma_0) k^{n+1}, \quad n = 0, 1, 2, \ldots, \quad x \in \mathbb{R}^+.
\end{equation}

Writing inequalities  \eqref{Khachatryan2.38} for indices  $n+1, n+2, \ldots, n+p-1$, then adding the obtained inequalities and inequality \eqref{Khachatryan2.38}, by virtue of  \eqref{Khachatryan2.37} we obtain:
\begin{equation}\label{Khachatryan2.39}
0 \leq B_n(x) - B_{n+p}(x) \leq \xi (1 - \sigma_0) k^{n+1}(1 + k + \ldots + k^p) \leq \frac{\xi (1 - \sigma_0) k^{n+1}}{1 - k},
\end{equation}
$n, p = 0, 1, 2, \ldots, x \in \mathbb{R}^+.$

Fixing the index $n$  in \eqref{Khachatryan2.39} and passing $p$ to infinity, we obtain the following estimate:
\begin{equation}\label{Khachatryan2.40}
0 \leq B_n(x) - B(x) \leq C \cdot k^{n+1}, \quad n = 0, 1, 2, \ldots, x \in \mathbb{R}^+,
\end{equation}
where
\begin{equation}\label{Khachatryan2.41}
C := \frac{\xi (1 - \sigma_0)}{1 - k}, \quad B(x): = \lim_{n \to \infty} B_n(x).
\end{equation}

From \eqref{Khachatryan2.22} and the uniform convergence of the sequence  $\{B_n(x)\}_{n=0}^\infty$ it follows that
\[
B \in C(\mathbb{R}^+).
\]

Taking into account the continuity of the functions  $K, \omega_2$, as well as  \eqref{Khachatryan2.23}, \eqref{Khachatryan2.24} and \eqref{Khachatryan2.40}, according to B. Levy's theorem we obtain that  $B(x)$ is a solution of equation  \eqref{Khachatryan1.2}. From \eqref{Khachatryan2.23}, \eqref{Khachatryan2.24}, \eqref{Khachatryan2.19} and \eqref{Khachatryan2.20} it also follows that:
\begin{equation}\label{Khachatryan2.42}
F(x) \leq B(x) \leq \Phi(x), \quad x \in \mathbb{R}^+, \quad \lim_{x \to +\infty} B(x) = \eta, \quad \eta - B \in L_1(\mathbb{R}^+).
\end{equation}
Thus, based on the above results, we arrive at the following theorem:
\begin{theorem}  Under conditions 1)--3), $b_1), b_2),$ and $ I)$--$III)$, equation \eqref{Khachatryan1.2} has a non-negative, nontrivial, bounded, and continuous on $\mathbb{R}^+$ solution $B(x)$, which possesses properties \eqref{Khachatryan2.40} and \eqref{Khachatryan2.42}.
\end{theorem}
\begin{Remark} By employing arguments similar to those used in \cite{khach14}, one can also prove that if, in addition, for some $p \in \mathbb{N}$, the integrals
$$
\int\limits_{0}^{\infty} x^p\mu_2(x) dx < +\infty, \quad \int\limits_{0}^{\infty} x^p  K(x) dx < +\infty,$$ are finite, then
$$ \int\limits_{0}^{\infty} x^{p-1} |\eta-B(x)|dx<+\infty.
$$
\end{Remark}
\section{Uniqueness of the Solution to Equation   \eqref{Khachatryan1.2}. Examples}
\subsection{Uniqueness Theorem for the Solution of Equation  \eqref{Khachatryan1.2}.}
\setcounter{equation}{0}
\begin{theorem} Under the assumptions of Theorem 2, equation \eqref{Khachatryan1.2} has only one solution in the class of functions
\begin{equation}\label{Khachatryan3.1}
\mathfrak{M} := \{B \in M(\mathbb{R}^+): B(x) \geq 0, x \in \mathbb{R}^+, \exists r > 0 \text{ s.t}\  \inf\limits_{x>r}B(x) > 0\}.
\end{equation}
\end{theorem}
\begin{proof} Let equation \eqref{Khachatryan1.2}, in addition to the solution $B \in \mathfrak{M}$ constructed using successive approximations \eqref{Khachatryan2.21}, also have another solution
$B^* \in \mathfrak{M}$. First, note that this solution is a continuous function on  $\mathbb{R}^+$, since $K \in C_M(\mathbb{R}) \cap L_1(\mathbb{R}), 0 \leq \omega_2(t,u) \leq \mu_2(t)u, \, (t, u) \in \mathbb{R}^+ \times \mathbb{R}^+, \, \mu_2 \in L_1(\mathbb{R}^+), \, B^* \in M(\mathbb{R}^+)$ and the convolution of the summable and bounded functions is a continuous function.

Let us now prove that the inequality holds:
\begin{equation}\label{Khachatryan3.2}
B^*(x) \leq B(x), \quad x \in \mathbb{R}^+.
\end{equation}
To this end, we recall that the unique solution $\Phi(x)$ of the auxiliary equation \eqref{Khachatryan2.18} is the pointwise limit of the following iterations (see  \cite{khach14}):
\begin{equation}\label{Khachatryan3.3}
\begin{array}{c}
\displaystyle Q(\Phi_{n+1}(x)) = \int\limits_{0}^\infty (K(x - t) - K(x + t)) \Phi_n(t) (1 + \mu_2(t)) dt, \\
\displaystyle \Phi_0(x) \equiv \xi, \, n = 0,1,2,\ldots, \, x \in \mathbb{R}^+.
\end{array}
\end{equation}

From above, by induction on $n$ we prove the validity of the following inequality:
\begin{equation}\label{Khachatryan3.4}
B^*(x) \leq \Phi_n(x), \quad n = 0, 1, 2, \ldots, \, x \in \mathbb{R}^+.
\end{equation}
Denote by  $C^*: = \sup\limits_{x\in\mathbb{R}^+} B^*(x) <+\infty$.  Then from equation \eqref{Khachatryan1.2} taking into account conditions $1), 2), b_1)$ and $ b_2)$ we have
\[
0 \leq Q(B^*(x)) \leq (1 + M) C^*, \quad x \in \mathbb{R}^+,
\]
whence, by virtue of properties  $I), II)$ for the nonlinearity of  $Q$ we have
\begin{equation}\label{Khachatryan3.5}
0 < Q(C^*) \leq C^*(1+M)= C^* \frac{Q(\xi)}{\xi}.
\end{equation}

Since the function  $\frac{Q(u)}{u}$ monotonically increases on  $(0, +\infty)$ (this property immediately follows from conditions $I), II)$), we obtain from  \eqref{Khachatryan3.5} that $C^* \leq \xi$. Consequently, inequality  \eqref{Khachatryan3.4} is satisfied for the number  $n=0$.

Let \eqref{Khachatryan3.4} holds for some  $n\in \mathbb{N}$. Then, taking into consideration conditions  $1), 2), b_1), b_2)$, inequality  \eqref{Khachatryan2.2} and recurrence relations  \eqref{Khachatryan3.3} we obtain
\[
0 \leq Q(B^*(x)) \leq  \int\limits_{0}^\infty (K(x - t) - K(x + t))(1 + \mu_2(t))B^*(t)dt\leq \]
\[\leq  \int\limits_{0}^\infty (K(x - t) - K(x + t))(1 + \mu_2(t)) \Phi_n(t) \, dt=Q(\Phi_{n+1}(x)),\quad x\in\mathbb{R}^+,
\]
whence, due to the monotonicity of the function  $Q$ we arrive at the inequality  $B^*(x)\leq \Phi_{n+1}(x),$ $x \in \mathbb{R}^+.$ In  \eqref{Khachatryan3.4} passing $n$ to infinity, we obtain
\begin{equation}\label{Khachatryan3.6}
B^*(x) \leq \Phi(x), \quad x \in \mathbb{R}^+.
\end{equation}
Now, by induction on  $n$ we show that
\begin{equation}\label{Khachatryan3.7}
B^*(x) \leq B_n(x), \quad n = 0, 1, 2, \ldots, \, x \in \mathbb{R}^+.
\end{equation}
For $n = 0$ inequality \eqref{Khachatryan3.7} follows immediately from  \eqref{Khachatryan3.6} and the definition of the zeroth approximation in iterations \eqref{Khachatryan2.21}. Let \eqref{Khachatryan3.7} be satisfied for some natural  $n$. Then, taking into consideration \eqref{Khachatryan2.2}, the monotonicity of the function  $\omega_2(t, u)$ in $u$ and the induction hypothesis, from \eqref{Khachatryan2.21} we obtain that the function
\[
Q(B^*(x)) \leq Q(B_{n+1}(x)), \quad x \in \mathbb{R}^+,
\]
whence, due to the monotonicity of the function  $Q$ it follows that  $B^*(x) \leq B_{n+1}(x), \, x \in \mathbb{R}^+$. In \eqref{Khachatryan3.7} passing $n \to \infty$, we arrive at inequality  \eqref{Khachatryan3.2}. Now we prove that
\begin{equation}\label{Khachatryan3.8}
B^*(x) > 0 \quad \text{at } x \in (0, +\infty).
\end{equation}
Indeed, taking into account the fact that  $B^* \in \mathfrak{M}$ from \eqref{Khachatryan1.2} by \eqref{Khachatryan2.2} and conditions $1), 2), b_1)$ we have:
\[
Q(B^*(x)) \geq \int\limits_{r}^\infty (K(x - t) - K(x + t)) \big(B^*(t) + \omega_2(t,B^*(t))\big) \, dt \geq
\]
\[
\geq \int\limits_{r}^\infty (K(x - t) - K(x + t)) B^*(t)dt\geq \biggl(\int\limits_{r}^\infty K(x-t)dt-\int\limits_{r}^\infty K(x+t)dt\biggr)\inf\limits_{t>r}B^*(t)=\]
\[=(\inf\limits_{t>r}B^*(t)) \biggl(\int\limits_{-\infty}^{x-r} K(y)dy-\int\limits_{x+r}^\infty K(y)dy\biggr)> (\inf\limits_{t>r}B^*(t)) \biggl(\int\limits_{-\infty}^{-r} K(y)dy-\int\limits_{r}^\infty K(y)dy\biggr)=0=Q(0),\]
at $x\in(0,+\infty)$ which implies  \eqref{Khachatryan3.8}.

Let us now consider the following function on the set  $(0, +\infty)$:
\begin{equation}\label{Khachatryan3.9}
\mathcal{L}(x) = \frac{\int\limits_0^{\infty} \left( K(x - t) - K(x + t) \right) \left( B^*(t) + \omega_2(t, B^*(t)) \right) dt}{\int\limits_0^{\infty} \left( K(x - t) - K(x + t) \right) \left( B(t) + \omega_2(t, B(t)) \right) dt}.
\end{equation}
It follows from conditions  $1), 2), b_1), b_2),$ inequalities  \eqref{Khachatryan2.2}, \eqref{Khachatryan3.2}, \eqref{Khachatryan3.8} and \eqref{Khachatryan2.42} that
\begin{equation}\label{Khachatryan3.10}
\mathcal{L} \in C(0, +\infty), \quad 0 < \mathcal{L}(x) \leq 1, \quad x \in (0, +\infty).
\end{equation}
Let us now study the limit value of the function $\mathcal{L}(x)$, when $x \to 0^+$. According to the theorem L'Hôpital, we have
\[1\geq \mathcal{L}(+0) =
\frac{\int\limits_0^{\infty} \left( K'(-t) - K'(t) \right) \left( B^*(t) + \omega_2(t, B^*(t)) \right) dt}{\int\limits_0^{\infty} \left( K'(-t) - K'(t) \right) \left( B(t) + \omega_2(t, B(t)) \right) dt}=\]
\[= \frac{-2 \int\limits_0^{\infty} K'(t) \left( B^*(t) + \omega_2(t, B^*(t)) \right) dt} {-2 \int\limits_0^{\infty} K'(t) \left( B(t) + \omega_2(t, B(t)) \right) dt}\geq  \frac{- \int\limits_0^{\infty} K'(t) \left( B^*(t) + \omega_2(t, B^*(t)) \right) dt} {-\xi \int\limits_0^{\infty} K'(t) \left( 1 + \mu_2(t)) \right) dt}=:\varepsilon^*>0.
\]
$K'(t)<0,$ $B^*(t)>0$ for $t>0$ and $\mu_2(t)\geq0$ on $\mathbb{R}^+.$

Let us now verify that  $\varepsilon^* < 1$. Indeed, to this end, we first prove that
the strict inequality holds:
\begin{equation}\label{Khachatryan3.11}
B(x) < \xi, \quad x \in \mathbb{R}^+.
\end{equation}

Using inequality  \eqref{Khachatryan2.2}, the monotonicity of the function  $Q(u)$ and $\omega_2(t, u)$ in $u$ and conditions   $1), 2),$ from \eqref{Khachatryan1.2} we have
\[
0 \leq Q(B(x)) \leq \int\limits_0^{\infty} \left( K(x - t) - K(x + t) \right) \left( \xi + \omega_2(t, \xi) \right) dt\leq\]
\[
 \leq \xi \int\limits_0^{\infty} \left( K(x - t) - K(x + t) \right) (1 + \mu_2(t)) dt <\] \[< \xi \int\limits_0^{\infty} K(x - t) (1 + \mu_2(t)) dt \leq \xi (1 + M) = Q(\xi), \quad x \in \mathbb{R}^+,
\]
whence \eqref{Khachatryan3.11} follows. Taking into account now \eqref{Khachatryan3.11} and \eqref{Khachatryan3.2}, we obtain
\[
\varepsilon^* \leq \frac{-\int\limits_0^{\infty} K'(t) \left( B(t) + \omega_2(t, B(t)) \right) dt}
{-\xi\int\limits_0^{\infty} K'(t) \left( 1 + \mu_2(t) \right) dt}<
\frac{-\int\limits_0^{\infty} K'(t) \left( \xi + \omega_2(t,\xi) \right) dt}
{-\xi\int\limits_0^{\infty} K'(t) \left(1 + \mu_2(t) \right) dt} \leq 1.
\]
Thus, we have proved that
\begin{equation}\label{Khachatryan3.12}
1\geq \mathcal{L}(+0)\geq \varepsilon^*,\quad \varepsilon^*\in(0,1).
\end{equation}
Based on  \eqref{Khachatryan3.10} and \eqref{Khachatryan3.12} the function  $\mathcal{L}(x)$ can be considered continuous on the set  $\mathbb{R}^+$ and it can be asserted that there exists a number  $\sigma^* \in (0, 1)$ such that
\begin{equation}\label{Khachatryan3.13}
\sigma^* \leq \mathcal{L}(x) \leq 1, \quad x \in \mathbb{R}^+.
\end{equation}
From \eqref{Khachatryan3.13} it immediately follows that
\begin{equation}\label{Khachatryan3.14}
\sigma^* Q(B(x)) \leq Q(B^*(x)) \leq Q(B(x)), \quad x \in \mathbb{R}^+.
\end{equation}
Next, using  \eqref{Khachatryan3.14} and making similar arguments as in the proof of Theorem  2, we conclude that there exist numbers $C^*>0$ and $k_* \in (0, 1)$ such that
\begin{equation}\label{Khachatryan3.15}
0 \leq B(x) - B^*(x) \leq C^* k_*^n, \quad n = 1, 2, \dots, \quad x \in \mathbb{R}^+.
\end{equation}
In \eqref{Khachatryan3.15} fixing  $x \in \mathbb{R}^+$ and passing  $n \to \infty$, we obtain  $B(x) = B^*(x), \, x \in \mathbb{R}^+$.

This completes the proof of the theorem.
\end{proof}
\begin{Remark} Theorem 3 established above generalizes the corresponding uniqueness result presented in \cite{khach17}.
\end{Remark}
\subsection{Examples.}
To conclude the paper, we present several illustrative examples of kernels $K$ and nonlinearities $Q$, $\omega_j, j = 1, 2$,  that satisfy all the conditions required by the theorems proved above. We begin with examples of the kernel $K:$
\begin{enumerate}
  \item [$k_1)$] $K(x) = \frac{1}{\sqrt{\pi}} e^{-x^2}, \quad x \in \mathbb{R},$
  \item [$k_2)$] $K(x) = \frac{\sqrt{2}}{\pi} \frac{1}{1 + x^4}, \quad x \in \mathbb{R}.$
\end{enumerate}

Now we give examples for the nonlinearity  $Q$:
\begin{enumerate}
  \item [$q_1)$] $Q(u) = u^p, \quad p > 1$--- \text{is a number}, $\quad u \in \mathbb{R}^+,$
  \item [$q_2)$] $Q(u) = \left(\frac{\sqrt{8u + 1} - 1}{2}\right)^{2p}, \quad p > \frac{3}{2}$--- $ \text{is a number}, \quad u \in \mathbb{R}^+.$
\end{enumerate}
Let us now turn to examples for the functions $\omega_j$, $j = 1, 2$:

\begin{enumerate}
  \item [$\Omega_1)$] $\omega_1(t, u) = e^{-t} \frac{u}{u + 1}, \quad t, u \in \mathbb{R}^+,$
  \item [$\Omega_2)$] $\omega_1(t, u) = \frac{1}{1 + t^3}(1 - e^{-u}), \quad t, u \in \mathbb{R}^+,$
  \item [$\Omega_3)$] $\omega_2(t, u) = e^{-t^2} \left( 1 - e^{-u} \right), \quad t, u \in \mathbb{R}^+,$
  \item [$\Omega_4)$]  $\omega_2(t, u) = \frac{1}{1 + t^2} \ln(1 + u), \quad t, u \in \mathbb{R}^+.$
\end{enumerate}
Let us stop in detail on the example  $q_2)$ and check the fulfillment of conditions I)–III) and \eqref{Khachatryan1.3}. First, from the representation of the function  $Q(u)$ it immediately follows that  $Q \in C(\mathbb{R}^+)$ and $Q(0) = 0$.  Since $Q'(u) = 2p \cdot \left(\frac{\sqrt{8u + 1} - 1}{2 }\right)^{2p-1} \frac{2}{\sqrt{8u + 1}} \geq 0, \, u \in \mathbb{R}^+$,  conditions  I) and \eqref{Khachatryan1.3} are satisfied.
Let us now prove that $Q(u)$  is a convex function on $\mathbb{R}^+$. Indeed, this follows directly from the following relations:
\[
Q''(u) = -\frac{16p}{(8u + 1)\sqrt{8u + 1}} \left( \frac{\sqrt{8u + 1} - 1}{2} \right)^{2p - 1}
+ \frac{4p}{\sqrt{8u + 1}} (2p - 1) \left( \frac{\sqrt{8u + 1} - 1}{2} \right)^{2p - 2}.
\]

\[
\cdot \frac{2}{\sqrt{8u + 1}} = \left( \frac{\sqrt{8u + 1} - 1}{2} \right)^{2p - 2}\frac{1}{8u+1} \cdot
\left(8p (2p - 1)-\frac{8p(\sqrt{8u + 1} - 1)}{\sqrt{8u + 1}}\right)\geq
\]

\[
\geq \left( \frac{\sqrt{8u + 1} - 1}{2} \right)^{2p - 2} \frac{8p}{8u + 1} (2p - 2) \geq 0, \quad u \in \mathbb{R}^+.
\]

Next, we verify that the characteristic equation $Q(u) = (1 + M) u, \, M > 0$ has a unique solution   $\xi > 1$. Since
\[
\frac{Q(u)}{u} \geq \frac{8u+2-2\sqrt{8u + 1}}{4u}\frac{\sqrt{8u + 1}-1}{2} \to +\infty \quad \text{at } u \to +\infty,
\]

then for the function  $\rho(u) := \frac{Q(u)}{u} - (1 + M)$ the following facts hold on the set $[1,+\infty)$
\[
\rho(1)=-M<0,\quad \rho(+\infty)=+\infty,\quad \rho \in C[1, +\infty).
\]
On the other hand, due to the convexity of the function $Q$, it is monotonically increasing on the set $[1, +\infty)$. Therefore, there exists a unique number $\xi > 1$  such that  $Q(\xi) = (1 + M)\xi$.
This confirms that condition II) is also satisfied.
Finally, we proceed to verify condition $III)$ for example $q_2).$ First, observe that in this case, the inverse function of $Q$ admits the following representation:
\begin{equation}\label{Khachatryan3.16}
G(u) = \frac{1}{2} \left(u^\alpha + u^{\frac{\alpha}{2}}\right), \quad \text{where } \alpha: = \frac{1}{p}, \; u \in \mathbb{R}^+.
\end{equation}
As a function  $\varphi(\sigma)$, if we choose  $\varphi(\sigma)=\sigma^\alpha,$ $\sigma\in[0,1],$ then by \eqref{Khachatryan3.16}, we will have
\[
G(\sigma u) = \frac{1}{2} \left(\sigma^\alpha u^\alpha + \sigma^{\frac{\alpha}{2}} u^{\frac{\alpha}{2}}\right) \geq \frac{1}{2} \left(\sigma^\alpha u^\alpha + \sigma^{\alpha}u^{\frac{\alpha}{2}}\right) =\sigma^\alpha G(u) = \varphi(\sigma)G(u), \ \sigma \in [0,1], u \in \mathbb{R}^+.
\]

Thus, condition $III)$ is satisfied, for example, in case $q_2)$. The verification of the corresponding conditions for the remaining examples can be carried out using similar arguments.

\begin{Remark}
It is worth noting that examples $k_1), q_1), \Omega_2)$, and $\Omega_3),$ in addition to their theoretical significance, also have practical applications in various areas of modern natural science (see, for instance, \cite{eng1}-\cite{cerc4}).
\end{Remark}

\textbf{Contact information:}

1)Aghavard Khachaturovich Khachatryan

Armenian National Agrarian University,

Teryan 74, Yerevan, 0009, Republic of Armenia

Department of Higher Mathematics, Physics and Applied Mechanics,

E-mail: Aghavard59@mail.ru

2) Khachatur Aghavardovich Khachatryan

Yerevan State University, 

1 Alex Manoogian, Yerevan, 0025, Republic of Armenia

Faculty of Mathematics and Mechanics,

E-mail: khachatur.khachatryan@ysu.am

3)Haykanush Samvelovna Petrosyan

Armenian National Agrarian University,

Teryan 74, Yerevan, 0009, Republic of Armenia

Department of Higher Mathematics, Physics and Applied Mechanics,

E-mail: Haykuhi25@mail.ru

%\newpage

\end{document}